\documentclass[12pt]{article}%
\usepackage{graphicx}
\usepackage{amsmath}
\usepackage{amsfonts}
\usepackage{amssymb}%
\providecommand{\U}[1]{\protect\rule{.1in}{.1in}}
\providecommand{\U}[1]{\protect\rule{.1in}{.1in}}
\providecommand{\U}[1]{\protect\rule{.1in}{.1in}}
\newtheorem{theorem}{Theorem}

\newtheorem{algorithm}[theorem]{Algorithm}

\newtheorem{condition}[theorem]{Condition}

\numberwithin{equation}{section}
\numberwithin{theorem}{section}
\newtheorem{definition}[theorem]{Definition}

\newtheorem{lemma}[theorem]{Lemma}

\newtheorem{remark}[theorem]{Remark}

\newenvironment{proof}[1][Proof]{\textbf{#1.} }{\ \rule{0.5em}{0.5em}}
\begin{document}

\title{\textbf{Algorithms for the Split Variational Inequality Problem}}
\author{Yair Censor$^{1}$, Aviv Gibali$^{2}$ and Simeon Reich$^{2}$\bigskip\\$^{1}$Department of Mathematics, University of Haifa,\\Mt.\ Carmel, 31905 Haifa, Israel \bigskip\\$^{2}$Department of Mathematics,\\The Technion - Israel Institute of Technology\\Technion City, 32000 Haifa, Israel}
\date{March 26, 2011. Revised: June 23, 2011.\medskip\\
This version contains some further corrections discovered at the galley
proof-reading stage.}
\maketitle

\begin{abstract}
We propose a prototypical Split Inverse Problem (SIP) and a new variational
problem, called the Split Variational Inequality Problem (SVIP), which is a
SIP. It entails finding a solution of one inverse problem (e.g., a Variational
Inequality Problem (VIP)), the image of which under a given bounded linear
transformation is a solution of another inverse problem such as a VIP. We
construct iterative algorithms that solve such problems, under reasonable
conditions, in Hilbert space and then discuss special cases, some of which are
new even in Euclidean space.\bigskip


\end{abstract}

\section{Introduction}

In\ this paper we introduce a new problem, which we call the \textit{Split
Variational Inequality Problem} (SVIP). The connection of SVIP to inverse
problems and many relevant references to earlier work are presented in Section
\ref{sec:sip}. Let $H_{1}$ and $H_{2}$ be two real Hilbert spaces$.$ Given
operators $f:H_{1}\rightarrow H_{1}$ and $g:H_{2}\rightarrow H_{2},$ a bounded
linear operator $A:H_{1}\rightarrow H_{2}$, and nonempty, closed and convex
subsets $C\subseteq H_{1}$ and $Q\subseteq H_{2},$ the SVIP is formulated as
follows:%
\begin{gather}
\text{find a point }x^{\ast}\in C\text{ such that }\left\langle f(x^{\ast
}),x-x^{\ast}\right\rangle \geq0\text{ for all }x\in C\label{eq:vip}\\
\text{and such that}\nonumber\\
\text{the point }y^{\ast}=Ax^{\ast}\in Q\text{ and solves }\left\langle
g(y^{\ast}),y-y^{\ast}\right\rangle \geq0\text{ for all }y\in Q.
\label{eq:svip}%
\end{gather}

When looked at separately, (\ref{eq:vip}) is the classical \textit{Variational
Inequality Problem} (VIP) and we denote its solution set by $SOL(C,f)$. The
SVIP constitutes a pair of VIPs, which have to be solved so that the image
$y^{\ast}=Ax^{\ast},$ under a given bounded linear operator $A,$ of the
solution $x^{\ast}$ of the VIP in $H_{1}$, is a solution of another VIP in
another space $H_{2}$.

SVIP is quite general and should enable split minimization between two spaces
so that the image of a solution point of one minimization problem, under a
given bounded linear operator, is a solution point of another minimization
problem. Another special case of the SVIP is the \textit{Split Feasibility
Problem} (SFP) which had already been studied and used in practice as a model
in intensity-modulated radiation therapy (IMRT) treatment planning; see
\cite{CBMT, CEKB}.

We consider two approaches to the solution of the SVIP. The first approach is
to look at the product space $H_{1}\times H_{2}$ and transform the SVIP
(\ref{eq:vip})--(\ref{eq:svip}) into an equivalent \textit{Constrained VIP}
(CVIP) in the product space. We study this CVIP and devise an iterative
algorithm for its solution, which becomes applicable to the original SVIP via
the equivalence between the problems. Our new iterative algorithm for the
CVIP, thus for the SVIP, is inspired by an extension of the extragradient
method of Korpelevich \cite{Korpelevich}. In the second approach we present a
method that does not require the translation to a product space. This
algorithm is inspired by the work of Censor and Segal \cite{CS08a} and Moudafi
\cite{Moudafi}.

Our paper is organized as follows. The connection of SVIP to inverse problems
and many relevant references to earlier work are presented in Section
\ref{sec:sip}. In Section \ref{sec:Preliminaries} we present some
preliminaries. In Section \ref{sec:Algorithm} the algorithm for the
constrained VIP is presented. In Section \ref{sec:SVIP} we analyze the SVIP
and present its equivalence with the CVIP in the product space. In Section
\ref{sec:Direct SVIP} we first present our method for solving the SVIP, which
does not rely on any product space formulation, and then prove convergence. In
Section \ref{sec:applications} we present some applications of the SVIP. It
turns out that in addition to helping us solve the SVIP, the CVIP unifies and
improves several existing problems and methods where a VIP has to be solved
with some additional constraints. Further relations of our results to
previously published work are discussed in detail after Theorems \ref{th:cvip}
and \ref{Theorem1}.

\section{\textbf{The split variational inequality problem as a methodology for
inverse problems\label{sec:sip}}}

Following the case which has already been studied and used in practice as a
model in intensity-modulated radiation therapy (IMRT) treatment planning; see
\cite{CBMT, CEKB}, a prototypical \textit{Split Inverse Problem}\textbf{
}(SIP)\ concerns a model in which there are two spaces $X$ and $Y$ and there
is given a bounded linear operator $A:X\rightarrow Y.$ Additionally, there are
two inverse problems involved, one inverse problem denoted IP$_{1}$ formulated
in the space $X$ and another inverse problem IP$_{2}$ formulated in the space
$Y.$ The Split Inverse Problem (SIP) is the following:%
\begin{gather}
\text{find a point }x^{\ast}\in X\text{ that solves IP}_{1}\text{ }\\
\text{such that }\nonumber\\
\text{the point }y^{\ast}=Ax^{\ast}\in Y\text{ solves IP}_{2}\text{.}%
\end{gather}

Many models of inverse problems can be cast in this framework by choosing
different inverse problems for IP$_{1}$ and IP$_{2}$. The Split Convex
Feasibility Problem (SCFP) first published in \textit{Numerical Algorithms}
\cite{CE} is the first instance of a SIP\ in which the two problems IP$_{1}$
and IP$_{2}$ are CFPs each. This was used for solving an inverse problem in
radiation therapy treatment planning in \cite{CEKB}. More work on the SCFP can
be found in \cite{byrne02, CEKB, Dan-Gao, Moudafi, qx05, ssl08, xu06, Xu,
yang04, zhl09, zy05}. Two candidates for IP$_{1}$ and IP$_{2}$ that come to
mind are the mathematical models of the Convex Feasibility Problem (CFP) and
the problem of constrained optimization. In particular, the CFP formalism is
in itself at the core of the modeling of many inverse problems in various
areas of mathematics and the physical sciences; see, e.g., \cite{cap88} and
references therein for an early example. Over the past four decades, the CFP
has been used to model significant real-world inverse problems in sensor
networks, in radiation therapy treatment planning, in resolution enhancement,
in wavelet-based denoising, in antenna design, in computerized tomography, in
materials science, in watermarking, in data compression, in demosaicking, in
magnetic resonance imaging, in holography, in color imaging, in optics and
neural networks, in graph matching and in adaptive filtering, see
\cite{cccdh11} for exact references to all the above. More work on the CFP can
be found in \cite{Byrne, byrne04, cdh10}.

It is therefore natural to investigate if other inversion models for IP$_{1}$
and IP$_{2}$, besides CFP, can be embedded in the SIP methodology. For
example, CFP in the space $X$ and constrained optimization in the space $Y$?
In this paper we make a step in this direction by formulating a SIP with
Variational Inequality Problems (VIP) in each of the two spaces of the SIP.
Since, as is well-known, both CFP and constrained optimization are special
cases of VIP, our newly-proposed SVIP covers the earlier SCFP and allows for
new SIP situations. Such new situations are described in Section
\ref{sec:applications} below.

\section{Preliminaries\label{sec:Preliminaries}}

Let $H$ be a real Hilbert space with inner product $\langle\cdot,\cdot\rangle$
and norm $\Vert\cdot\Vert,$ and let $D$ be a nonempty, closed and convex
subset of $H$. We write $x^{k}\rightharpoonup x$ to indicate that the sequence
$\left\{  x^{k}\right\}  _{k=0}^{\infty}$ converges weakly to $x,$ and
$x^{k}\rightarrow x$ to indicate that the sequence $\left\{  x^{k}\right\}
_{k=0}^{\infty}$ converges strongly to $x.$ For every point $x\in H,$\ there
exists a unique nearest point in $D$, denoted by $P_{D}(x)$. This point
satisfies%
\begin{equation}
\left\Vert x-P_{D}\left(  x\right)  \right\Vert \leq\left\Vert x-y\right\Vert
\text{\textit{ }for all}\mathit{\ }y\in D.
\end{equation}
The mapping $P_{D}$ is called the metric projection of $H$ onto $D$. We know
that $P_{D}$ is a nonexpansive operator of $H$ onto $D$, i.e.,%
\begin{equation}
\left\Vert P_{D}\left(  x\right)  -P_{D}\left(  y\right)  \right\Vert
\leq\left\Vert x-y\right\Vert \text{\textit{ }for all}\mathit{\ }x,y\in H.
\end{equation}
The metric projection $P_{D}$ is characterized by the fact that $P_{D}\left(
x\right)  \in D$ and
\begin{equation}
\left\langle x-P_{D}\left(  x\right)  ,P_{D}\left(  x\right)  -y\right\rangle
\geq0\text{ for all }x\in H,\text{ }y\in D, \label{eq:ProjP1}%
\end{equation}
and has the property%
\begin{equation}
\left\Vert x-y\right\Vert ^{2}\geq\left\Vert x-P_{D}\left(  x\right)
\right\Vert ^{2}+\left\Vert y-P_{D}\left(  x\right)  \right\Vert ^{2}\text{
for all }x\in H,\text{ }y\in D. \label{eq:ProjP2}%
\end{equation}
It is known that in a Hilbert space $H$,%
\begin{equation}
\Vert\lambda x+(1-\lambda)y\Vert^{2}=\lambda\Vert x\Vert^{2}+(1-\lambda)\Vert
y\Vert^{2}-\lambda(1-\lambda)\Vert x-y\Vert^{2} \label{eq:ConvexComb}%
\end{equation}
for all $x,y\in H$ and $\lambda\in\lbrack0,1].$

The following lemma was proved in \cite[Lemma 3.2]{Takahashi}.

\begin{lemma}
\label{Lemma:Takahashi} Let $H$ be a Hilbert space and let $D$ be a nonempty,
closed and convex subset of $H.$ If the sequence $\left\{  x^{k}\right\}
_{k=0}^{\infty}\subset H$ is \texttt{Fej\'{e}r-monotone} with respect to $D,$
i.e., for every $u\in D,$%
\begin{equation}
\Vert x^{k+1}-u\Vert\leq\Vert x^{k}-u\Vert\text{ for all }k\geq0,
\end{equation}
then $\left\{  P_{D}\left(  x^{k}\right)  \right\}  _{k=0}^{\infty}$ converges
strongly to some $z\in D.$
\end{lemma}

The next lemma is also known (see, e.g., \cite[Lemma 3.1]{Nadezhkina}).

\begin{lemma}
\label{Lemma:Schu} Let $H$ be a Hilbert space, $\left\{  \alpha_{k}\right\}
_{k=0}^{\infty}$ be a real sequence satisfying $0<a\leq\alpha_{k}\leq b<1$ for
all $k\geq0,$ and let $\left\{  v^{k}\right\}  _{k=0}^{\infty}$ and $\left\{
w^{k}\right\}  _{k=0}^{\infty}$ be two sequences in $H$ such that for some
$\sigma\geq0$,%
\begin{equation}
\limsup_{k\rightarrow\infty}\Vert v^{k}\Vert\leq\sigma,\text{ and }%
\limsup_{k\rightarrow\infty}\Vert w^{k}\Vert\leq\sigma.
\end{equation}
If%
\begin{equation}
\lim_{k\rightarrow\infty}\Vert\alpha_{k}v^{k}+(1-\alpha_{k})w^{k}\Vert=\sigma,
\end{equation}
then%
\begin{equation}
\lim_{k\rightarrow\infty}\Vert v^{k}-w^{k}\Vert=0.
\end{equation}

\end{lemma}

\begin{definition}
Let $H$ be a Hilbert space, $D$ a closed and convex subset of $H,$ and let
$M:D\rightarrow H$ be an operator. Then $M$ is said to be \texttt{demiclosed}
at $y\in H$ if for any sequence $\left\{  x^{k}\right\}  _{k=0}^{\infty}$ in
$D$ such that $x^{k}\rightharpoonup\overline{x}\in D$ and $M(x^{k})\rightarrow
y,$ we have $M(\overline{x})=y.$
\end{definition}

Our next lemma is the well-known Demiclosedness Principle \cite{Browder}.

\begin{lemma}
Let $H$ be a Hilbert space, $D$ a closed and convex subset of $H,$ and
$N:D\rightarrow H$ a nonexpansive operator. Then $I-N$ ($I$ is the identity
operator on $H$) is \texttt{demiclosed} at $y\in H.$
\end{lemma}

For instance, the orthogonal projection $P$ onto a closed and convex set is a
demiclosed operator everywhere because $I-P$ is nonexpansive \cite[page
17]{Goebel+Reich}.

The next property is known as the \textit{Opial condition} \cite[Lemma
1]{Opial}. It characterizes the weak limit of a weakly convergent sequence in
Hilbert space.

\begin{condition}
(\textbf{Opial) }For any sequence $\left\{  x^{k}\right\}  _{k=0}^{\infty}$ in
$H$ that converges weakly to $x$,%
\begin{equation}
\liminf_{k\rightarrow\infty}\Vert x^{k}-x\Vert<\liminf_{k\rightarrow\infty
}\Vert x^{k}-y\Vert\text{ for all }y\neq x.
\end{equation}

\end{condition}

\begin{definition}
Let $h:H\rightarrow H$ be an operator and let $D\subseteq H.$

(i) $h$ is called \texttt{inverse strongly monotone (ISM)} with constant
$\alpha$ on $D\subseteq H$ if%
\begin{equation}
\langle h(x)-h(y),x-y\rangle\geq\alpha\Vert h(x)-h(y)\Vert^{2}\text{ for all
}x,y\in D.
\end{equation}

(ii) $h$ is called \texttt{monotone} on $D\subseteq H$ if%
\begin{equation}
\langle h(x)-h(y),x-y\rangle\geq0\text{ for all }x,y\in D.
\end{equation}

\end{definition}

\begin{definition}
An operator $h:H\rightarrow H$ is called \texttt{Lipschitz continuous} on
$D\subseteq H$ with constant $\kappa>0$ if%
\begin{equation}
\Vert h(x)-h(y)\Vert\leq\kappa\Vert x-y\Vert\text{\ for all\ }x,y\in D.
\end{equation}

\end{definition}

\begin{definition}
Let $S:H\rightrightarrows2^{H}\mathcal{\ }$be a point-to-set operator defined
on a real Hilbert space $H$. $S$ is called a \texttt{maximal monotone
operator} if $S$ is \texttt{monotone}, i.e.,%
\begin{equation}
\left\langle u-v,x-y\right\rangle \geq0,\text{ for all }u\in S(x)\text{ and
for all }v\in S(y),
\end{equation}
and the graph $G(S)$ of $S,$%
\begin{equation}
G(S):=\left\{  \left(  x,u\right)  \in H\times H\mid u\in S(x)\right\}  ,
\end{equation}
is not properly contained in the graph of any other monotone operator.
\end{definition}

It is clear that a monotone operator $S$ is maximal if and only if, for each
$\left(  x,u\right)  \in H\times H,$ $\left\langle u-v,x-y\right\rangle \geq0$
for all $\left(  v,y\right)  \in G(S)$ implies that $u\in S(x).$

\begin{definition}
Let $D$ be a nonempty, closed and convex subset of $H.$ The \texttt{normal
cone} of $D$ at the point $w\in D$ is defined by%
\begin{equation}
N_{D}\left(  w\right)  :=\{d\in H\mid\left\langle d,y-w\right\rangle
\leq0\text{ for all }y\in D\}. \label{eq:normal-c}%
\end{equation}

\end{definition}

Let $h$ be an $\alpha$-ISM operator on $D\subseteq H,$ define the following
point-to-set operator:%
\begin{equation}
S(w):=\left\{
\begin{array}
[c]{cc}%
h(w)+N_{D}\left(  w\right)  , & w\in C,\\
\emptyset, & w\notin C.
\end{array}
\right.  \label{eq:maximal-S}%
\end{equation}

In these circumstances, it follows from \cite[Theorem 3]{Rockafellar76} that
$S$ is maximal monotone. In addition, $0\in S(w)\ $if and only if $w\in$
$SOL(D,h).$

For $T:H\rightarrow H$, denote by $\operatorname*{Fix}(T)$ the fixed point set
of $T,$ i.e.,%
\begin{equation}
\operatorname*{Fix}(T):=\{x\in H\mid T(x)=x\}.
\end{equation}
It is well-known that%
\begin{equation}
x^{\ast}\in SOL(D,h)\Leftrightarrow x^{\ast}=P_{D}(x^{\ast}-\lambda h(x^{\ast
})), \label{eq:fix-vip}%
\end{equation}
i.e., $x^{\ast}\in\operatorname*{Fix}(P_{D}(I-\lambda h)).$ It is also known
that every nonexpansive operator $T:H\rightarrow H$ satisfies, for all
$(x,y)\in H\times H,$ the inequality
\begin{equation}
\langle(x-T(x))-(y-T(y)),T(y)-T(x)\rangle\leq(1/2)\Vert(T(x)-x)-(T(y)-y)\Vert
^{2}\text{ }%
\end{equation}
and therefore we get, for all $(x,y)\in H\times\operatorname*{Fix}(T),$
\begin{equation}
\langle x-T(x),y-T(x)\rangle\leq(1/2)\Vert T(x)-x\Vert^{2};\text{ }
\label{eq:Ne(Crombez)}%
\end{equation}
see, e.g., \cite[Theorem 3]{Crombez06} and \cite[Theorem 1]{Crombez}.

In the next lemma we collect several important properties that will be needed
in the sequel.

\begin{lemma}
\label{lemma:Mod-proj} Let $D\subseteq H$ be a nonempty, closed and convex
subset and let $h:H\rightarrow H$ be an $\alpha$-ISM operator on $H$. If
$\lambda\in\lbrack0,2\alpha],$ then\smallskip\ 

(i) the operator $P_{D}(I-\lambda h)$ is nonexpansive on $D.$

If, in addition, for all $x^{\ast}\in SOL(D,h),$%
\begin{equation}
\langle h(x),P_{D}(I-\lambda h)(x)-x^{\ast}\rangle\geq0\text{ for all\ }x\in
H, \label{eq:2.24}%
\end{equation}
then$\smallskip$ the following inequalities hold:

(ii) for all $x\in H$ and $q\in\operatorname*{Fix}(P_{D}(I-\lambda h)),$%
\begin{equation}
\langle P_{D}(I-\lambda h)(x)-x,P_{D}(I-\lambda h)(x)-q\rangle\leq0;
\label{QFNE}%
\end{equation}

(iii) for all $x\in H$ and $q\in\operatorname*{Fix}(P_{D}(I-\lambda h)),$%
\begin{equation}
\left\Vert P_{D}(I-\lambda h)(x)-q\right\Vert ^{2}\leq\left\Vert
x-q\right\Vert ^{2}-\left\Vert P_{D}(I-\lambda h)(x)-x\right\Vert ^{2}.
\label{eq:2.26}%
\end{equation}

\end{lemma}

\begin{proof}
(i) Let $x,y\in H.$ Then%
\begin{align}
\Vert P_{D}(I-\lambda h)(x)-P_{D}(I-\lambda h)(y)\Vert^{2}  &  =\Vert
P_{D}(x-\lambda h(x))-P_{D}(y-\lambda h(y))\Vert^{2}\nonumber\\
&  \leq\Vert x-\lambda h(x)-(y-\lambda h(y))\Vert^{2}\nonumber\\
&  =\Vert(x-y)-\lambda(h(x)-h(y))\Vert^{2}\nonumber\\
&  =\Vert x-y\Vert^{2}-2\lambda\langle x-y,h(x)-h(y)\rangle\nonumber\\
&  +\lambda^{2}\Vert h(x)-h(y)\Vert^{2}\nonumber\\
&  \leq\Vert x-y\Vert^{2}-2\lambda\alpha\Vert h(x)-h(y)\Vert^{2}\nonumber\\
&  +\lambda^{2}\Vert h(x)-h(y)\Vert^{2}\nonumber\\
&  =\Vert x-y\Vert^{2}+\lambda(\lambda-2\alpha)\Vert h(x)-h(y)\Vert
^{2}\nonumber\\
&  \leq\Vert x-y\Vert^{2}.
\end{align}
(ii) Let $x\in H$ and $q\in\operatorname*{Fix}(P_{D}(I-\lambda h)).$ Then%
\begin{align}
&  \langle P_{D}(x-\lambda h(x))-x,P_{D}(x-\lambda h(x))-q\rangle\nonumber\\
&  =\langle P_{D}(x-\lambda h(x))-x+\lambda h(x)-\lambda h(x),P_{D}(x-\lambda
h(x))-q\rangle\nonumber\\
&  =\langle P_{D}(x-\lambda h(x))-(x-\lambda h(x)),P_{D}(x-\lambda
h(x))-q\rangle\nonumber\\
&  -\lambda\langle h(x),P_{D}(x-\lambda h(x))-q\rangle.
\end{align}
By (\ref{eq:ProjP1}), (\ref{eq:fix-vip}) and (\ref{eq:2.24}), we get%
\begin{equation}
\langle P_{D}(x-\lambda h(x))-x,P_{D}(x-\lambda h(x))-q\rangle\leq0.
\end{equation}
(iii) Let $x\in H$ and $q\in\operatorname*{Fix}(P_{D}(I-\lambda h)).$ Then%
\begin{align}
\left\Vert q-x\right\Vert ^{2}  &  =\left\Vert (P_{D}(I-\lambda
h)(x)-x)-(P_{D}(I-\lambda h)(x)-q)\right\Vert ^{2}\nonumber\\
&  =\left\Vert P_{D}(I-\lambda h)(x)-x\right\Vert ^{2}+\left\Vert
P_{D}(I-\lambda h)(x)-q\right\Vert ^{2}\nonumber\\
&  -2\langle P_{D}(I-\lambda h)(x)-x,P_{D}(I-\lambda h)(x)-q\rangle.
\end{align}
By (ii), we get%
\begin{equation}
-2\langle P_{D}(I-\lambda h)(x)-x,P_{D}(I-\lambda h)(x)-q\rangle\geq0.
\end{equation}
Thus,%
\begin{equation}
\left\Vert q-x\right\Vert ^{2}\geq\left\Vert P_{D}(I-\lambda
h)(x)-x\right\Vert ^{2}+\left\Vert P_{D}(I-\lambda h)(x)-q\right\Vert ^{2}%
\end{equation}
or%
\begin{equation}
\left\Vert P_{D}(I-\lambda h)(x)-q\right\Vert ^{2}\leq\left\Vert
q-x\right\Vert ^{2}-\left\Vert P_{D}(I-\lambda h)(x)-x\right\Vert ^{2},
\end{equation}
as asserted.
\end{proof}

Observe that, under the additional condition (\ref{eq:2.24}), Equation
(\ref{QFNE}) means that the operator $P_{D}(I-\lambda h)$ belongs to the class
of operators called the $\mathcal{T}$-class. This class $\mathcal{T}$ of
operators was introduced and investigated by Bauschke and Combettes in
\cite[Definition 2.2]{BC01} and by Combettes in \cite{Co}. Operators in this
class were named \textit{directed operators }by Zaknoon \cite{Z} and further
studied under this name by Segal \cite{Seg08} and by Censor and Segal
\cite{CS08, CS08a, CS09}. Cegielski \cite[Def. 2.1]{Ceg08} studied these
operators under the name \textit{separating operators}. Since both
\textit{directed }and\textit{\ separating }are key words of other,
widely-used, mathematical entities, Cegielski and Censor have recently
introduced the term \textit{cutter operators} \cite{cc11}. This class
coincides with the class $\mathcal{F}^{\nu}$ for $\nu=1$ \cite{Crombez} and
with the class DC$_{\boldsymbol{p}}$ for $\boldsymbol{p}=-1$ \cite{mp08}. The
term \textit{firmly quasi-nonexpansive} (FQNE) for $\mathcal{T}$-class
operators was used by Yamada and Ogura \cite{Yamada} because every
\textit{firmly nonexpansive} (FNE) mapping \cite[page 42]{Goebel+Reich} is
obviously FQNE.

\section{An algorithm for solving the constrained variational inequality
problem\label{sec:Algorithm}}

Let\textit{ }$f:H\rightarrow H$, and let $C$ and $\Omega$ be nonempty, closed
and convex subsets of $H$. The \textit{Constrained} \textit{Variational
Inequality Problem} (CVIP) is:
\begin{equation}
\text{find }x^{\ast}\in C\cap\Omega\text{ such that }\left\langle f(x^{\ast
}),x-x^{\ast}\right\rangle \geq0\text{ for all }x\in C. \label{eq:cvip}%
\end{equation}
The iterative algorithm for this CVIP, presented next, is inspired by our
earlier work \cite{CGR,CGR2} in which we modified the extragradient method of
Korpelevich \cite{Korpelevich}. The following conditions are needed for the
convergence theorem.

\begin{condition}
\label{Condition:a} $f$ is monotone on $C$.
\end{condition}

\begin{condition}
\label{Condition:b} $f$ is Lipschitz continuous on $H$ with constant
$\kappa>0.$
\end{condition}

\begin{condition}
\label{Condition:c} $\Omega\cap SOL(C,f)\neq\emptyset.$
\end{condition}

Let $\left\{  \lambda_{k}\right\}  _{k=0}^{\infty}\subset\left[  a,b\right]
$\ for some $a,b\in(0,1/\kappa)$, and let \textit{ }$\left\{  \alpha
_{k}\right\}  _{k=0}^{\infty}\subset\left[  c,d\right]  $ for some\textit{
}$c,d\in(0,1)$. Then the following algorithm generates two sequences that
converge to a point $z\in\Omega$ $\cap$ SOL$(C,f),$ as the convergence theorem
that follows shows.

\begin{algorithm}
\label{alg:SubExt4SVIP}$\left.  {}\right.  $

\textbf{Initialization:} Select an arbitrary starting point $x^{0}\in H$.

\textbf{Iterative step:} Given the current iterate $x^{k},$ compute%
\begin{equation}
y^{k}=P_{C}(x^{k}-\lambda_{k}f(x^{k})),
\end{equation}
construct the half-space $T_{k}$ the bounding hyperplane of which supports $C$
at $y^{k},$%
\begin{equation}
T_{k}:=\{w\in H\mid\left\langle \left(  x^{k}-\lambda_{k}f(x^{k})\right)
-y^{k},w-y^{k}\right\rangle \leq0\},
\end{equation}
and then calculate the next iterate by%
\begin{equation}
x^{k+1}=\alpha_{k}x^{k}+(1-\alpha_{k})P_{\Omega}\left(  P_{T_{k}}%
(x^{k}-\lambda_{k}f(y^{k}))\right)  . \label{eq:3.4}%
\end{equation}

\end{algorithm}

\begin{theorem}
\label{th:cvip}Let\textit{ }$f:H\rightarrow H$, and let $C$ and $\Omega$ be
nonempty, closed and convex subsets of $H$. Assume that Conditions
\ref{Condition:a}--\ref{Condition:c} hold, and let $\left\{  x^{k}\right\}
_{k=0}^{\infty}$ and $\left\{  y^{k}\right\}  _{k=0}^{\infty}$ be any two
sequences generated by Algorithm \ref{alg:SubExt4SVIP} with $\left\{
\lambda_{k}\right\}  _{k=0}^{\infty}\subset\left[  a,b\right]  $\textit{\ for
some }$a,b\in(0,1/\kappa)$\textit{ and }$\left\{  \alpha_{k}\right\}
_{k=0}^{\infty}\subset\left[  c,d\right]  $ for some\textit{ }$c,d\in(0,1)$.
Then $\left\{  x^{k}\right\}  _{k=0}^{\infty}$ and $\left\{  y^{k}\right\}
_{k=0}^{\infty}$ converge weakly to the same point $z\in\Omega\cap SOL(C,f)$
and%
\begin{equation}
z=\lim_{k\rightarrow\infty}P_{\Omega\cap SOL(C,f)}(x^{k}).
\end{equation}

\end{theorem}

\begin{proof}
For the special case of fixed $\lambda_{k}=\tau$ for all $k\geq0$ this theorem
is a direct consequence of our \cite[Theorem 7.1]{CGR2} with the choice of the
nonexpansive operator $S$ there to be $P_{\Omega}$. However, a careful
inspection of the proof of \cite[Theorem 7.1]{CGR2} reveals that it also
applies to a variable sequence $\left\{  \lambda_{k}\right\}  _{k=0}^{\infty}$
as used here.
\end{proof}

To relate our results to some previously published works we mention two lines
of research related to our notion of the CVIP. Takahashi and Nadezhkina
\cite{Nadezhkina} proposed an algorithm for finding a point $x^{\ast}%
\in\operatorname*{Fix}(N)\cap$SOL$(C,f),$ where $N:C\rightarrow C$ is a
nonexpansive operator. The iterative step of their algorithm is as
follows.\textbf{ }Given the current iterate $x^{k},$ compute%
\begin{equation}
y^{k}=P_{C}(x^{k}-\lambda_{k}f(x^{k}))
\end{equation}
and then%
\begin{equation}
x^{k+1}=\alpha_{k}x^{k}+(1-\alpha_{k})N\left(  P_{C}(x^{k}-\lambda_{k}%
f(y^{k}))\right)  .
\end{equation}
The restriction $P_{\Omega}|_{C}$ of our $P_{\Omega}$ in (\ref{eq:3.4}) is, of
course, nonexpansive, and so it is a special case of $N$ in \cite{Nadezhkina}.
But a significant advantage of our Algorithm \ref{alg:SubExt4SVIP} lies in the
fact that we compute $P_{T_{k}}$ onto a half-space in (\ref{eq:3.4}) whereas
the authors of \cite{Nadezhkina} need to project onto the convex set $C.$
Various ways have been proposed in the literature to cope with the inherent
difficulty of calculating projections (onto closed convex sets) that do not
have a closed-form expression; see, e.g., He, Yang and Duan \cite{hyd10}, or
\cite{cgr-oms}.

Bertsekas and Tsitsiklis \cite[Page 288]{BT} consider the following problem in
Euclidean space: given $f:R^{n}\rightarrow R^{n}$, polyhedral sets
$C_{1}\subset R^{n}$ and $C_{2}\subset R^{m},$ and an $m\times n$ matrix $A$,
find a point $x^{\ast}\in C_{1}$ such that $Ax^{\ast}\in C_{2}$ and%
\begin{equation}
\left\langle f(x^{\ast}),x-x^{\ast}\right\rangle \geq0\text{ for all }x\in
C_{1}\cap\{y\mid Ay\in C_{2}\}.
\end{equation}

Denoting $\Omega=A^{-1}(C_{2})$, we see that this problem becomes similar to,
but not identical with a CVIP. While the authors of \cite{BT} seek a solution
in SOL$(C_{1}\cap\Omega,f),$ we aim in our CVIP at $\Omega\cap$SOL$(C,f).$
They propose to solve their problem by the method of multipliers, which is a
different approach than ours, and they need to assume that either $C_{1}$ is
bounded or $A^{t}A$ is invertible, where $A^{t}$ is the transpose of $A.$

\section{The split variational inequality problem as a constrained variational
inequality problem in a product space\label{sec:SVIP}}

Our first approach to the solution of the SVIP (\ref{eq:vip})--(\ref{eq:svip})
is to look at the product space $\boldsymbol{H}=H_{1}\times H_{2}$ and
introduce in it the product set $\boldsymbol{D}:=C\times Q$ and the set%
\begin{equation}
\boldsymbol{V}:=\{\mathbf{x}=(x,y)\in\boldsymbol{H}\mid Ax=y\}.
\end{equation}
We adopt the notational convention that objects in the product space are
represented in boldface type. We transform the SVIP (\ref{eq:vip}%
)--(\ref{eq:svip}) into the following equivalent CVIP in the product space:
\begin{align}
\text{Find a point }\boldsymbol{x}^{\ast}  &  \in\boldsymbol{D}\cap
\boldsymbol{V},\text{ such that }\left\langle \boldsymbol{h}(\boldsymbol{x}%
^{\ast}),\boldsymbol{x}-\boldsymbol{x}^{\ast}\right\rangle \geq0\text{
}\nonumber\\
\text{for all }\boldsymbol{x}  &  =(x,y)\in\boldsymbol{D},
\label{eq:c-as-svip}%
\end{align}
where $\boldsymbol{h}:\boldsymbol{H}\rightarrow\boldsymbol{H}$ is defined by%
\begin{equation}
\boldsymbol{h}(x,y)=(f(x),g(y)).
\end{equation}
A simple adaptation of the decomposition lemma \cite[Proposition 5.7, page
275]{BT} shows that problems (\ref{eq:vip})--(\ref{eq:svip}) and
(\ref{eq:c-as-svip}) are equivalent, and, therefore, we can apply Algorithm
\ref{alg:SubExt4SVIP} to the solution of (\ref{eq:c-as-svip}).

\begin{lemma}
A point $\boldsymbol{x}^{\ast}=(x^{\ast},y^{\ast})$ solves (\ref{eq:c-as-svip}%
) if and only if $x^{\ast}$ and $y^{\ast}$ solve (\ref{eq:vip}%
)--(\ref{eq:svip}).
\end{lemma}

\begin{proof}
If $(x^{\ast},y^{\ast})$ solves (\ref{eq:vip})--(\ref{eq:svip}), then it is
clear that $(x^{\ast},y^{\ast})$ solves (\ref{eq:c-as-svip}). To prove the
other direction, suppose that $(x^{\ast},y^{\ast})$ solves (\ref{eq:c-as-svip}%
). Since (\ref{eq:c-as-svip}) holds for all $(x,y)\in\boldsymbol{D}$, we may
take $(x^{\ast},y)\in\boldsymbol{D}$ and deduce that%
\begin{equation}
\left\langle g(Ax^{\ast}),y-Ax^{\ast}\right\rangle \geq0\text{ for all }y\in
Q.
\end{equation}
Using a similar argument with $(x,y^{\ast})\in\boldsymbol{D,}$ we get%
\begin{equation}
\left\langle f(x^{\ast}),x-x^{\ast}\right\rangle \geq0\text{ for all }x\in C,
\end{equation}
which means that $(x^{\ast},y^{\ast})$ solves (\ref{eq:vip})--(\ref{eq:svip}).
\end{proof}

Using this equivalence, we can now employ Algorithm \ref{alg:SubExt4SVIP} in
order to solve the SVIP. The following conditions are needed for the
convergence theorem.

\begin{condition}
\label{Condition:a2} $f$ is monotone on $C$ and $g$ is monotone on $Q$.
\end{condition}

\begin{condition}
\label{Condition:b2} $f$ is Lipschitz continuous on $H_{1}$ with constant
$\kappa_{1}>0$ and $g$ is Lipschitz continuous on $H_{2}$ with constant
$\kappa_{2}>0.$
\end{condition}

\begin{condition}
\label{Condition:c2} $\boldsymbol{V}\cap SOL(\boldsymbol{D},\boldsymbol{h}%
)\neq\emptyset.$
\end{condition}

Let $\left\{  \lambda_{k}\right\}  _{k=0}^{\infty}\subset\left[  a,b\right]
$\ for some $a,b\in(0,1/\kappa)$, where $\kappa=\min\{\kappa_{1},\kappa_{2}%
\}$, and let $\left\{  \alpha_{k}\right\}  _{k=0}^{\infty}\subset\left[
c,d\right]  $ for some\textit{ }$c,d\in(0,1)$. Then the following algorithm
generates two sequences that converge to a point $\boldsymbol{z}%
\in\boldsymbol{V}\cap SOL(\boldsymbol{D},\boldsymbol{h}),$ as the convergence
theorem given below shows.

\begin{algorithm}
\label{alg:SubExt4SVIP2}$\left.  {}\right.  $

\textbf{Initialization:} Select an arbitrary starting point $\boldsymbol{x}%
^{0}\in\boldsymbol{H}$.

\textbf{Iterative step:} Given the current iterate $\boldsymbol{x}^{k},$
compute%
\begin{equation}
\boldsymbol{y}^{k}=\boldsymbol{P}_{\boldsymbol{D}}(\boldsymbol{x}^{k}%
-\lambda_{k}\boldsymbol{h}(\boldsymbol{x}^{k})),
\end{equation}
construct the half-space $\boldsymbol{T}_{k}$ the bounding hyperplane of which
supports $\boldsymbol{D}$ at $\boldsymbol{y}^{k},$%
\begin{equation}
\boldsymbol{T}_{k}:=\{\boldsymbol{w}\in\boldsymbol{H}\mid\left\langle \left(
\boldsymbol{x}^{k}-\lambda_{k}\boldsymbol{h}(\boldsymbol{x}^{k})\right)
-\boldsymbol{y}^{k},\boldsymbol{w}-\boldsymbol{y}^{k}\right\rangle \leq0\},
\end{equation}
and then calculate%
\begin{equation}
\boldsymbol{x}^{k+1}=\alpha_{k}\boldsymbol{x}^{k}+(1-\alpha_{k})\boldsymbol{P}%
_{\boldsymbol{V}}\left(  \boldsymbol{P}_{\boldsymbol{T}_{k}}(\boldsymbol{x}%
^{k}-\lambda_{k}\boldsymbol{h}(\boldsymbol{y}^{k}))\right)  . \label{eq:5.8}%
\end{equation}

\end{algorithm}

Our convergence theorem for Algorithm \ref{alg:SubExt4SVIP2} follows from
Theorem \ref{th:cvip}.

\begin{theorem}
Consider $f:H_{1}\rightarrow H_{1}$ and $g:H_{2}\rightarrow H_{2},$ a bounded
linear operator $A:H_{1}\rightarrow H_{2}$, and nonempty, closed and convex
subsets $C\subseteq H_{1}$ and $Q\subseteq H_{2}$. Assume that Conditions
\ref{Condition:a2}--\ref{Condition:c2} hold, and let $\left\{  \boldsymbol{x}%
^{k}\right\}  _{k=0}^{\infty}$ and $\left\{  \boldsymbol{y}^{k}\right\}
_{k=0}^{\infty}$ be any two sequences generated by Algorithm
\ref{alg:SubExt4SVIP2} with $\left\{  \lambda_{k}\right\}  _{k=0}^{\infty
}\subset\left[  a,b\right]  $\textit{\ for some }$a,b\in(0,1/\kappa)$\textit{,
where }$\kappa=\min\{\kappa_{1},\kappa_{2}\}$, \textit{and let} $\left\{
\alpha_{k}\right\}  _{k=0}^{\infty}\subset\left[  c,d\right]  $ for
some\textit{ }$c,d\in(0,1)$. Then $\left\{  \boldsymbol{x}^{k}\right\}
_{k=0}^{\infty}$ and $\left\{  \boldsymbol{y}^{k}\right\}  _{k=0}^{\infty}$
converge weakly to the same point $\boldsymbol{z}\in\boldsymbol{V}\cap
SOL(\boldsymbol{D},\boldsymbol{h})$ and%
\begin{equation}
\boldsymbol{z}=\lim_{k\rightarrow\infty}\boldsymbol{P}_{\boldsymbol{V}\cap
SOL(\boldsymbol{D},\boldsymbol{h})}(\boldsymbol{x}^{k}).
\end{equation}

\end{theorem}

The value of the product space approach, described above, depends on the
ability to \textquotedblleft translate\textquotedblright\ Algorithm
\ref{alg:SubExt4SVIP2} back to the original spaces $H_{1}$ and $H_{2}.$
Observe that due to \cite[Lemma 1.1]{Pierra} for $\boldsymbol{x}%
\mathbf{=}(x,y)\in\boldsymbol{D,}$ we have $\boldsymbol{P}_{\boldsymbol{D}%
}(\boldsymbol{x})=(P_{C}(x),P_{Q}(y))$ and a similar formula holds for
$\boldsymbol{P}_{\boldsymbol{T}_{k}}.$ The potential difficulty lies in
$\boldsymbol{P}_{\boldsymbol{V}}$ of (\ref{eq:5.8}). In the finite-dimensional
case, since $\boldsymbol{V}$ is a subspace, the projection onto it is easily
computable by using an orthogonal basis. For example, if $U$ is a
$k$-dimensional subspace of $R^{n}$ with the basis $\{u_{1},u_{2},...,u_{k}%
\}$, then for $x\in R^{n},$ we have%
\begin{equation}
P_{U}(x)=\sum\limits_{i=1}^{k}\frac{\left\langle x,u_{i}\right\rangle }{\Vert
u_{i}\Vert^{2}}u_{i}.
\end{equation}

\section{Solving the split variational inequality problem without a product
space\label{sec:Direct SVIP}}

In this section we present a method\ for solving the SVIP, which does not need
a product space formulation as in the previous section. Recalling that
$SOL(C,f)$ and $SOL(Q,g)$ are the solution sets of (\ref{eq:vip}) and
(\ref{eq:svip}), respectively, we see that the solution set of the SVIP is%
\begin{equation}
\Gamma:=\Gamma(C,Q,f,g,A):=\left\{  z\in SOL(C,f)\mid Az\in SOL(Q,g)\right\}
.
\end{equation}
Using the abbreviations $T:=P_{Q}(I-\lambda g)$ and $U:=P_{C}(I-\lambda f),$
we propose the following algorithm.

\begin{algorithm}
\label{Alg-SVIP}$\left.  {}\right.  $

\textbf{Initialization:} Let $\lambda>0$ and select an arbitrary starting
point $x^{0}\in H_{1}$.

\textbf{Iterative step:} Given the current iterate $x^{k},$ compute%
\begin{equation}
x^{k+1}=U(x^{k}+\gamma A^{\ast}(T-I)(Ax^{k})),
\end{equation}
where $\gamma\in(0,1/L)$, $L$ is the spectral radius of the operator $A^{\ast
}A$, and $A^{\ast}$ is the adjoint of $A$.
\end{algorithm}

The following lemma, which asserts Fej\'{e}r-monotonicity, is crucial for the
convergence theorem.

\begin{lemma}
\label{lemma:Fejer} Let $H_{1}$ and $H_{2}$ be real Hilbert spaces and let
$A:H_{1}\rightarrow H_{2}$ be a bounded linear operator. Let $f:H_{1}%
\rightarrow H_{1}$ and $g:H_{2}\rightarrow H_{2}$ be\ $\alpha_{1}$-ISM and
$\alpha_{2}$-ISM operators on $H_{1}$ and $H_{2},$ respectively, and set
$\alpha:=\min\{\alpha_{1},\alpha_{2}\}$. Assume that $\Gamma\neq\emptyset$ and
that $\gamma\in(0,1/L)$. Consider the operators $U=P_{C}(I-\lambda f)$ and
$T=P_{Q}(I-\lambda g)$ with $\lambda\in\lbrack0,2\alpha]$. Then any sequence
$\left\{  x^{k}\right\}  _{k=0}^{\infty},$ generated by Algorithm
\ref{Alg-SVIP}, is Fej\'{e}r-monotone with respect to the solution set
$\Gamma$.
\end{lemma}

\begin{proof}
Let $z\in\Gamma.$ Then $z\in SOL(C,f)$ and, therefore, by (\ref{eq:fix-vip})
and Lemma \ref{lemma:Mod-proj}(i), we get%
\begin{align}
\left\Vert x^{k+1}-z\right\Vert ^{2}  &  =\left\Vert U\left(  x^{k}+\gamma
A^{\ast}(T-I)(Ax^{k})\right)  -z\right\Vert ^{2}\nonumber\\
&  =\left\Vert U\left(  x^{k}+\gamma A^{\ast}(T-I)(Ax^{k})\right)
-U(z)\right\Vert ^{2}\nonumber\\
&  \leq\left\Vert x^{k}+\gamma A^{\ast}(T-I)(Ax^{k})-z\right\Vert
^{2}\nonumber\\
&  =\left\Vert x^{k}-z\right\Vert ^{2}+\gamma^{2}\left\Vert A^{\ast
}(T-I)(Ax^{k})\right\Vert ^{2}\nonumber\\
&  +2\gamma\left\langle x^{k}-z,A^{\ast}(T-I)(Ax^{k})\right\rangle .
\end{align}
Thus%
\begin{align}
\left\Vert x^{k+1}-z\right\Vert ^{2}  &  \leq\left\Vert x^{k}-z\right\Vert
^{2}+\gamma^{2}\left\langle (T-I)(Ax^{k}),AA^{\ast}(T-I)(Ax^{k})\right\rangle
\nonumber\\
&  +2\gamma\left\langle x^{k}-z,A^{\ast}(T-I)(Ax^{k})\right\rangle .
\label{P0}%
\end{align}
From the definition of $L$ it follows, by standard manipulations, that
\begin{align}
\gamma^{2}\left\langle (T-I)(Ax^{k}),AA^{\ast}(T-I)(Ax^{k})\right\rangle  &
\leq L\gamma^{2}\left\langle (T-I)(Ax^{k}),(T-I)(Ax^{k})\right\rangle
\nonumber\\
&  =L\gamma^{2}\left\Vert (T-I)(Ax^{k})\right\Vert ^{2}. \label{P1}%
\end{align}
Denoting $\Theta:=2\gamma\left\langle x^{k}-z,A^{\ast}(T-I)(Ax^{k}%
)\right\rangle $ and using (\ref{eq:Ne(Crombez)}), we obtain%
\begin{align}
\Theta &  =2\gamma\left\langle A(x^{k}-z),(T-I)(Ax^{k})\right\rangle
\nonumber\\
&  =2\gamma\left\langle A(x^{k}-z)+(T-I)(Ax^{k})-(T-I)(Ax^{k}),(T-I)(Ax^{k}%
)\right\rangle \nonumber\\
&  =2\gamma\left(  \left\langle T(Ax^{k})-Az,(T-I)(Ax^{k})\right\rangle
-\left\Vert (T-I)(Ax^{k})\right\Vert ^{2}\right) \nonumber\\
&  \leq2\gamma\left(  (1/2)\left\Vert (T-I)(Ax^{k})\right\Vert ^{2}-\left\Vert
(T-I)(Ax^{k})\right\Vert ^{2}\right) \nonumber\\
&  \leq-\gamma\left\Vert (T-I)(Ax^{k})\right\Vert ^{2}. \label{P2}%
\end{align}
Applying (\ref{P1}) and (\ref{P2}) to (\ref{P0}), we see that%
\begin{equation}
\left\Vert x^{k+1}-z\right\Vert ^{2}\leq\left\Vert x^{k}-z\right\Vert
^{2}+\gamma(L\gamma-1)\left\Vert (T-I)(Ax^{k})\right\Vert ^{2}. \label{P3}%
\end{equation}
From the definition of $\gamma,$ we get%
\begin{equation}
\left\Vert x^{k+1}-z\right\Vert ^{2}\leq\left\Vert x^{k}-z\right\Vert ^{2},
\label{eq:x_k-z}%
\end{equation}
which completes the proof.
\end{proof}

Now we present our convergence result for Algorithm \ref{Alg-SVIP}.

\begin{theorem}
\label{Theorem1} Let $H_{1}$ and $H_{2}$ be real Hilbert spaces and let
$A:H_{1}\rightarrow H_{2}$ be a bounded linear operator. Let $f:H_{1}%
\rightarrow H_{1}$ and $g:H_{2}\rightarrow H_{2}$ be\ $\alpha_{1}$-ISM and
$\alpha_{2}$-ISM operators on $H_{1}$ and $H_{2},$ respectively, and set
$\alpha:=\min\{\alpha_{1},\alpha_{2}\}$. Assume that $\gamma\in(0,1/L)$.
Consider the operators $U=P_{C}(I-\lambda f)$ and $T=P_{Q}(I-\lambda g)$ with
$\lambda\in\lbrack0,2\alpha]$. Assume further that $\Gamma\neq\emptyset$ and
that, for all $x^{\ast}\in SOL(C,f),$%
\begin{equation}
\langle f(x),P_{C}(I-\lambda f)(x)-x^{\ast}\rangle\geq0\text{ for all\ }x\in
H_{1}. \label{eq:5.9}%
\end{equation}
Then any sequence $\left\{  x^{k}\right\}  _{k=0}^{\infty},$ generated by
Algorithm \ref{Alg-SVIP}, converges weakly to a solution point $x^{\ast}%
\in\Gamma$.
\end{theorem}

\begin{proof}
Let $z\in\Gamma.$ It follows from (\ref{eq:x_k-z}) that the sequence $\left\{
\left\Vert x^{k}-z\right\Vert \right\}  _{k=0}^{\infty}$ is monotonically
decreasing and therefore convergent, which shows, by (\ref{P3}), that,%
\begin{equation}
\lim_{k\rightarrow\infty}\left\Vert (T-I)(Ax^{k})\right\Vert =0. \label{P4}%
\end{equation}
Fej\'{e}r-monotonicity implies that $\left\{  x^{k}\right\}  _{k=0}^{\infty}$
is bounded, so it has a weakly convergent subsequence $\left\{  x^{k_{j}%
}\right\}  _{j=0}^{\infty}$ such that $x^{k_{j}}\rightharpoonup x^{\ast}$. By
the assumptions on $\lambda$ and $g,$ we get from Lemma \ref{lemma:Mod-proj}%
(i) that $T$ is nonexpansive. Applying the demiclosedness of $T-I$ at $0$ to
(\ref{P4}), we obtain%
\begin{equation}
T(Ax^{\ast})=Ax^{\ast}, \label{P5}%
\end{equation}
which means that $Ax^{\ast}\in SOL(Q,g)$. Denote%
\begin{equation}
u^{k}:=x^{k}+\gamma A^{\ast}(T-I)(Ax^{k})\text{.}%
\end{equation}
Then%
\begin{equation}
u^{k_{j}}=x^{k_{j}}+\gamma A^{\ast}(T-I)(Ax^{k_{j}}).
\end{equation}
Since $x^{k_{j}}\rightharpoonup x^{\ast},$ (\ref{P4}) implies that $u^{k_{j}%
}\rightharpoonup x^{\ast}$ too. It remains to be shown that $x^{\ast}\in
SOL(C,f)$. Assume, by negation, that $x^{\ast}\notin SOL(C,f),$ i.e.,
$Ux^{\ast}\neq x^{\ast}.$ By the assumptions on $\lambda$ and $f,$ we get from
Lemma \ref{lemma:Mod-proj}(i) that $U$ is nonexpansive and, therefore, $U-I$
is demiclosed at $0$. So, the negation assumption must lead to%
\begin{equation}
\lim_{j\rightarrow\infty}\left\Vert U(u^{k_{j}})-u^{k_{j}}\right\Vert \neq0.
\end{equation}
Therefore, there exists an $\varepsilon>0$ and a subsequence $\left\{
u^{k_{j_{s}}}\right\}  _{s=0}^{\infty}$ of $\left\{  u^{k_{j}}\right\}
_{j=0}^{\infty}$ such that%
\begin{equation}
\left\Vert U(u^{k_{j_{s}}})-u^{k_{j_{s}}}\right\Vert >\varepsilon\text{ for
all }s\geq0. \label{formula morethanDelta}%
\end{equation}
Condition (\ref{eq:5.9}) justifies the use of Lemma \ref{lemma:Mod-proj} by
supplying (\ref{eq:2.24}). Therefore, inequality (\ref{eq:2.26}) now yields,
for all $s\geq0,$%
\begin{align}
\left\Vert U(u^{k_{j_{s}}})-U(z)\right\Vert ^{2}  &  =\left\Vert
U(u^{k_{j_{s}}})-z\right\Vert ^{2}\leq\left\Vert u^{k_{j_{s}}}-z\right\Vert
^{2}-\left\Vert U(u^{k_{j_{s}}})-u^{k_{j_{s}}}\right\Vert ^{2}\nonumber\\
&  <\left\Vert u^{k_{j_{s}}}-z\right\Vert ^{2}-\varepsilon^{2}.
\label{formula withsqdelta}%
\end{align}
By arguments similar to those in the proof of Lemma \ref{lemma:Fejer}, we have%
\begin{equation}
\left\Vert u^{k}-z\right\Vert =\left\Vert \left(  x^{k}+\gamma A^{\ast
}(T-I)(Ax^{k})\right)  -z\right\Vert \leq\left\Vert x^{k}-z\right\Vert .
\label{P6}%
\end{equation}
Since $U$ is nonexpansive,%
\begin{equation}
\left\Vert x^{k+1}-z\right\Vert =\left\Vert U(u^{k})-z\right\Vert
\leq\left\Vert u^{k}-z\right\Vert . \label{P7}%
\end{equation}
Combining (\ref{P6}) and (\ref{P7}), we get%
\begin{equation}
\left\Vert x^{k+1}-z\right\Vert \leq\left\Vert u^{k}-z\right\Vert
\leq\left\Vert x^{k}-z\right\Vert , \label{P8}%
\end{equation}
which means that the sequence $\{x^{1},u^{1},x^{2},u^{2},\ldots\}$ is
Fej\'{e}r-monotone with respect to $\Gamma.$ Since $x^{k_{j_{s+1}}%
}=U(u^{k_{j_{s}}})$, we obtain%
\begin{equation}
\left\Vert u^{k_{j_{s+1}}}-z\right\Vert ^{2}\leq\left\Vert u^{k_{j_{s}}%
}-z\right\Vert ^{2}.
\end{equation}
Hence $\left\{  u^{k_{j_{s}}}\right\}  _{s=0}^{\infty}$ is also
Fej\'{e}r-monotone with respect to $\Gamma.$ Now, (\ref{formula withsqdelta})
and (\ref{P8}) imply that%
\begin{equation}
\left\Vert u^{k_{j_{s+1}}}-z\right\Vert ^{2}<\left\Vert u^{k_{j_{s}}%
}-z\right\Vert ^{2}-\varepsilon^{2}\text{ for all }s\geq0,
\end{equation}
which leads to a contradiction. Therefore $x^{\ast}\in SOL(C,f)$ and finally,
$x^{\ast}\in\Gamma$. Since the subsequence$\ \left\{  x^{k_{j}}\right\}
_{j=0}^{\infty}$ was arbitrary, we get that $x^{k}\rightharpoonup x^{\ast}.$
\end{proof}

Relations of our results to some previously published works are as follows. In
\cite{CS08a} an algorithm for the Split Common Fixed Point Problem (SCFPP) in
Euclidean spaces was studied. Later Moudafi \cite{Moudafi} presented a similar
result for Hilbert spaces. In this connection, see also \cite{Masad+Reich}.

To formulate the SCFPP, let $H_{1}$ and $H_{2}$ be two real Hilbert spaces.
Given operators $U_{i}:H_{1}\rightarrow H_{1}$, $i=1,2,\ldots,p,$ and
$T_{j}:H_{2}\rightarrow H_{2},$ $j=1,2,\ldots,r,$ with nonempty fixed point
sets $C_{i},$ $i=1,2,\ldots,p,$ and $Q_{j},$ $j=1,2,\ldots,r,$ respectively,
and a bounded linear operator $A:H_{1}\rightarrow H_{2}$, the SCFPP is
formulated as follows:%
\begin{equation}
\text{find a point }x^{\ast}\in C:=\cap_{i=1}^{p}C_{i}\text{ such that
}Ax^{\ast}\in Q:=\cap_{j=1}^{r}Q_{j}.
\end{equation}

Our result differs from\ those in \cite{CS08a} and \cite{Moudafi} in several
ways. Firstly, the spaces in which the problems are formulated. Secondly, the
operators $U$ and $T$ in \cite{CS08a} are assumed to be firmly
quasi-nonexpansive (FQNE; see the comments after Lemma \ref{lemma:Mod-proj}
above), where in our case here only $U$ is FQNE, while $T$ is just
nonexpansive. Lastly, Moudafi \cite{Moudafi} obtains weak convergence for a
wider class of operators, called demicontractive. The iterative step of his
algorithm is
\begin{equation}
x^{k+1}=(1-\alpha_{k})u^{k}+\alpha_{k}U(u^{k}),
\end{equation}
where $u^{k}:=x^{k}+\gamma A^{\ast}(T-I)(Ax^{k})$ for $\alpha_{k}\in(0,1).$ If
$\alpha_{k}=1,$ which is not allowed there, were possible, then the iterative
step of \cite{Moudafi} would coincide with that of \cite{CS08a}.

\subsection{A parallel algorithm for solving the multiple set split
variational inequality problem}

We extend the SVIP to the \textit{Multiple Set Split Variational Inequality
Problem} (MSSVIP), which is formulated as follows. Let $H_{1}$ and $H_{2}$ be
two real Hilbert spaces. Given a bounded linear operator $A:H_{1}\rightarrow
H_{2}$, functions $f_{i}:H_{1}\rightarrow H_{1},$ $i=1,2,\ldots,p,$ and
$g_{j}:H_{2}\rightarrow H_{2},$ $j=1,2,\ldots,r$, and nonempty, closed and
convex subsets $C_{i}\subseteq H_{1},$ $Q_{j}\subseteq H_{2}$ for
$i=1,2,\ldots,p$ and $j=1,2,\ldots,r$, respectively, the Multiple Set Split
Variational Inequality Problem (MSSVIP) is formulated as follows:%
\begin{equation}
\left\{
\begin{array}
[c]{l}%
\text{find a point }x^{\ast}\in C:=\cap_{i=1}^{p}C_{i}\text{ such that
}\left\langle f_{i}(x^{\ast}),x-x^{\ast}\right\rangle \geq0\text{ for all
}x\in C_{i}\\
\text{and for all }i=1,2,\ldots,p, \text{ and such that}\\
\text{the point }y^{\ast}=Ax^{\ast}\in Q:=\cap_{i=1}^{r}Q_{j}\text{ solves
}\left\langle g_{j}(y^{\ast}),y-y^{\ast}\right\rangle \geq0\text{ for all
}y\in Q_{j}\text{ }\\
\text{and for all }j=1,2,\ldots,r.
\end{array}
\right.  \label{eq:mssvip}%
\end{equation}

For the MSSVIP we do not yet have a solution approach which does not use a
product space formalism. Therefore we present a simultaneous algorithm for the
MSSVIP the analysis of which is carried out via a certain product space. Let
$\Psi$ be the solution set of the MSSVIP:%
\begin{equation}
\Psi:=\left\{  z\in\cap_{i=1}^{p}SOL(C_{i},f_{i})\mid Az\in\cap_{i=1}%
^{r}SOL(Q_{j},g_{j})\right\}  .
\end{equation}
We introduce the spaces $\boldsymbol{W}_{1}:\mathbf{=}H_{1}$ and
$\boldsymbol{W}_{2}:=H_{1}^{p}\times H_{2}^{r},$ where $r$ and $p$ are the
indices in (\ref{eq:mssvip}). Let $\left\{  \alpha_{i}\right\}  _{i=1}^{p}$
and $\left\{  \beta_{j}\right\}  _{j=1}^{r}$ be positive real numbers. Define
the following sets in their respective spaces:%
\begin{align}
\boldsymbol{C}  &  \mathbf{:}=H_{1}\text{ \ and \ \ \label{eq:prod}}\\
\boldsymbol{Q}  &  \mathbf{:}=\left(  \prod_{i=1}^{p}\sqrt{\alpha_{i}}%
C_{i}\right)  \times\left(  \prod_{j=1}^{r}\sqrt{\beta_{j}}Q_{j}\right)  ,
\end{align}
and the operator
\begin{equation}
\boldsymbol{A}\mathbf{:}=\left(  \sqrt{\alpha_{1}}I,\ldots,\sqrt{\alpha_{p}%
}I,\sqrt{\beta_{1}}A^{\ast},\ldots,\sqrt{\beta_{r}}A^{\ast}\right)  ^{\ast},
\end{equation}
where $A^{\ast}$ stands for adjoint of $A$. Denote $U_{i}:=P_{C_{i}}(I-\lambda
f_{i})$ and $T_{j}:=P_{Q_{j}}(I-\lambda g_{j})$ for $i=1,2,\ldots,p$ and
$j=1,2,\ldots,r$, respectively$.$ Define the operator $\boldsymbol{T}%
:\boldsymbol{W}_{2}\mathbf{\rightarrow}\boldsymbol{W}_{2}$ by%
\begin{align}
\boldsymbol{T}\mathbf{(}\boldsymbol{y}\mathbf{)}  &  =\boldsymbol{T}\left(
\begin{array}
[c]{c}%
y_{1}\\
y_{2}\\
\vdots\\
y_{p+r}%
\end{array}
\right) \nonumber\\
&  =\left(  \left(  U_{1}\left(  y_{1}\right)  \right)  ^{\ast},\ldots,\left(
U_{p}\left(  y_{p}\right)  \right)  ^{\ast},\left(  T_{1}\left(
y_{p+1}\right)  \right)  ^{\ast},\ldots,\left(  T_{r}(y_{p+r})\right)  ^{\ast
}\right)  ^{\ast}, \label{eq:fat}%
\end{align}
where $y_{1},y_{2},...,y_{p}\in H_{1}$ and $y_{p+1},y_{p+2},...,y_{p+r}\in
H_{2}$.

This leads to an SVIP with just two operators $\boldsymbol{F}$ and
$\boldsymbol{G}$ and two sets $\boldsymbol{C}$ and $\boldsymbol{Q},$
respectively, in the product space, when we take $\boldsymbol{C}%
\mathbf{=}H_{1}$, $\boldsymbol{F}\equiv\boldsymbol{0},$ $\boldsymbol{Q}%
\mathbf{\subseteq}\boldsymbol{W}_{2}$, $\boldsymbol{G}\mathbf{(}%
\boldsymbol{y}\mathbf{)}=\left(  f_{1}(y_{1}),f_{2}(y_{2})\ldots,f_{p}%
(y_{p}),g_{1}(y_{p+1}),g_{2}(y_{p+2}),\ldots,g_{r}(y_{p+r})\right)  ,$ and the
operator $\boldsymbol{A}:H_{1}\mathbf{\rightarrow}\boldsymbol{W}_{2}$. It is
easy to verify that the following equivalence holds:%
\begin{equation}
x\in\Psi\text{ if and only if }\boldsymbol{A}x\in\boldsymbol{Q}\mathbf{.}%
\end{equation}
Therefore we may apply Algorithm \ref{Alg-SVIP},%
\begin{equation}
x^{k+1}=x^{k}+\gamma\boldsymbol{A}^{\ast}(\boldsymbol{T}-\boldsymbol{I}%
)(\boldsymbol{A}x^{k})\text{ for all }k\geq0, \label{itstepinPS}%
\end{equation}
to the problem (\ref{eq:prod})--(\ref{eq:fat}) in order to obtain a solution
of the original MSSVIP. We translate the iterative step (\ref{itstepinPS}) to
the original spaces $H_{1}$ and $H_{2}$ using the relation%
\begin{equation}
\boldsymbol{T}\mathbf{(}\boldsymbol{A}x)=\left(  \sqrt{\alpha_{1}}%
U_{1}(x),\ldots,\sqrt{\alpha_{p}}U_{p}(x),\sqrt{\beta_{1}}AT_{1}%
(x),\ldots,\sqrt{\beta_{r}}AT_{r}(x)\right)  ^{\ast}%
\end{equation}
and obtain the following algorithm.

\begin{algorithm}
\label{Nirits alg}$\left.  {}\right.  $

\textbf{Initialization:}$\ $Select an arbitrary starting point $x^{0}\in
H_{1}$.

\textbf{Iterative step: }Given the current iterate $x^{k},$ compute%
\begin{equation}
x^{k+1}=x^{k}+\gamma\left(  \sum_{i=1}^{p}\alpha_{i}(U_{i}-I)(x^{k}%
)+\sum_{j=1}^{r}\beta_{j}A^{\ast}(T_{j}-I)(Ax^{k})\right)  ,
\end{equation}
where $\gamma\in(0,1/L),$ with $L=\sum_{i=1}^{p}\alpha_{i}+\sum_{j=1}^{r}%
\beta_{j}\Vert A\Vert^{2}$.
\end{algorithm}

The following convergence result follows from Theorem \ref{Theorem1}.

\begin{theorem}
\label{Theorem1*} Let $H_{1}$ and $H_{2}$ be two real Hilbert spaces and let
$A:H_{1}\rightarrow H_{2}$ be a bounded linear operator. Let $f_{i}%
:H_{1}\rightarrow H_{1},$ $i=1,2,\ldots,p,$ and $g_{j}:H_{2}\rightarrow
H_{2},$ $j=1,2,\ldots,r$, be $\alpha$-ISM operators on nonempty, closed and
convex subsets $C_{i}\subseteq H_{1},$ $Q_{j}\subseteq H_{2}$ for
$i=1,2,\ldots,p,$ and $j=1,2,\ldots,r$, respectively. Assume that\textbf{
}$\gamma\in(0,1/L)$ and $\Psi\neq\emptyset$\textbf{.} Set $U_{i}:=P_{C_{i}%
}(I-\lambda f_{i})$ and $T_{j}:=P_{Q_{j}}(I-\lambda g_{j})$ for $i=1,2,\ldots
,p$ and $j=1,2,\ldots,r$, respectively, with $\lambda\in\lbrack0,2\alpha]$.
If, in addition, for each $i=1,2,\ldots,p$ and $j=1,2,\ldots,r$ we have
\begin{equation}
\langle f_{i}(x),P_{C_{i}}(I-\lambda f_{i})(x)-x^{\ast}\rangle\geq0\text{ for
all\ }x\in H_{1} \label{eq:cond1}%
\end{equation}
for all $x^{\ast}\in SOL(C_{i},f_{i})$ and
\begin{equation}
\langle g_{j}(x),P_{Q_{j}}(I-\lambda g_{j})(x)-x^{\ast}\rangle\geq0\text{ for
all\ }x\in H_{2}, \label{eq:cond2}%
\end{equation}
for all $x^{\ast}\in SOL(C_{i},f_{i})$, then any sequence $\left\{
x^{k}\right\}  _{k=0}^{\infty},$ generated by Algorithm \ref{Alg-SVIP},
converges weakly to a solution point $x^{\ast}\in\Psi$.
\end{theorem}

\begin{proof}
Apply Theorem \ref{Theorem1} to the two-operator SVIP in the product space
setting with $U=\boldsymbol{I}:H_{1}\rightarrow H_{1}$, $\operatorname*{Fix}%
U=\boldsymbol{C}\mathbf{,}$ $T=\boldsymbol{T}:\boldsymbol{W}\rightarrow
\boldsymbol{W}\mathbf{,}$ and $\operatorname*{Fix}T=\boldsymbol{Q}$.
\end{proof}

\begin{remark}
Observe that conditions (\ref{eq:cond1}) and (\ref{eq:cond2}) imposed on
$U_{i}$ and $T_{j}$ for $i=1,2,\ldots,p$ and $j=1,2,\ldots,r$, respectively,
in Theorem \ref{Theorem1*}, which are necessary for our treatment of the
problem in a product space, ensure that these operators are firmly
quasi-nonexpansive (FQNE). Therefore, the SVIP under these conditions may be
considered a \texttt{Split Common Fixed Point Problem} (SCFPP), first
introduced in \cite{CS08a}, with $\boldsymbol{C}\mathbf{,}$ $\boldsymbol{Q}%
\mathbf{,}$ $\boldsymbol{A}$ and $\boldsymbol{T}:\boldsymbol{W}_{2}%
\rightarrow\boldsymbol{W}_{2}$ as above, and the identity operator
$\boldsymbol{I}:\boldsymbol{C}\rightarrow\boldsymbol{C}$.\textbf{ }Therefore,
we could also apply \cite[Algorithm 4.1]{CS08a}. If, however, we drop these
conditions, then the operators are nonexpansive, by Lemma \ref{lemma:Mod-proj}%
(i), and the result of \cite{Moudafi} would apply.
\end{remark}

\section{Applications\label{sec:applications}}

The following problems are special cases of the SVIP. They are listed here
because their analysis can benefit from our algorithms for the SVIP and
because known algorithms for their solution may be generalized in the future
to cover the more general SVIP. The list includes known problems such as the
Split Feasibility Problem (SFP) and the Convex Feasibility Problem\textit{
}(CFP). In addition, we introduce two new \textquotedblleft
split\textquotedblright\ problems that have, to the best of our knowledge,
never been studied before. These are the Common Solutions to Variational
Inequalities Problem (CSVIP) and the Split Zeros Problem (SZP).

\subsection{The split feasibility and convex feasibility problems}

The Split Feasibility Problem (SFP) in Euclidean space is formulated as
follows:
\begin{equation}
\text{find a point }x^{\ast}\text{ such that }x^{\ast}\in C\subseteq
R^{n}\text{ and }Ax^{\ast}\in Q\subseteq R^{m}, \label{eq:sfp}%
\end{equation}
where $C\subseteq R^{n},$ $Q\subseteq R^{m}$ are nonempty, closed and convex
sets, and $A:R^{n}\rightarrow R^{m}$ is given. Originally introduced in Censor
and Elfving \cite{CE}, it was later used in the area of intensity-modulated
radiation therapy (IMRT) treatment planning; see \cite{CEKB,CBMT}. Obviously,
it is formally a special case of the SVIP obtained from (\ref{eq:vip}%
)--(\ref{eq:svip}) by setting $f\equiv g\equiv0.$ The Convex Feasibility
Problem\textit{ }(CFP) in a Euclidean space\textit{ }is:
\begin{equation}
\text{find a point }x^{\ast}\text{ such that }x^{\ast}\in\cap_{i=1}^{m}%
C_{i}\neq\emptyset,
\end{equation}
where $C_{i},$ $i=1,2,\ldots,m,$ are nonempty, closed and convex sets in
$R^{n}.$ This, in its turn, becomes a special case of the SFP by taking in
(\ref{eq:sfp}) $n=m,$ $A=I$ $Q=R^{n}$ and $C=\cap_{i=1}^{m}C_{i}.$ Many
algorithms for solving the CFP have been developed; see, e.g., \cite{BB96,
CZ97}. Byrne \cite{Byrne} established an algorithm for solving the
SFP\textit{, }called the CQ-Algorithm, with the following iterative step:%
\begin{equation}
x^{k+1}=P_{C}\left(  x^{k}+\gamma A^{t}(P_{Q}-I)Ax^{k}\right)  ,
\label{eq:cq-alg}%
\end{equation}
which does not require calculation of the inverse of the operator $A,$ as in
\cite{CE}, but needs only its transpose $A^{t}$. A recent excellent paper on
the multiple-sets SFP which contains many references that reflect the
state-of-the-art in this area is \cite{lopezetal10}.

It is of interest to note that looking at the SFP from the point of view of
the SVIP enables us to find the minimum-norm solution of the SFP, i.e., a
solution of the form%
\begin{equation}
x^{\ast}=\operatorname{argmin}\{\Vert x\Vert\mid x\text{ solves the SFP
(\ref{eq:sfp})}\}.
\end{equation}
This is done, and easily verified, by solving (\ref{eq:vip})--(\ref{eq:svip})
with $f=I$ and $g\equiv0.$

\subsection{The common solutions to variational inequalities problem}

The Common Solutions to Variational Inequalities Problem (CSVIP), newly
introduced here, is defined in Euclidean space as follows. Let $\left\{
f_{i}\right\}  _{i=1}^{m}\ $be a family of functions from $R^{n}$ into itself
and let $\left\{  C_{i}\right\}  _{i=1}^{m}$ be nonempty, closed and convex
subsets of $R^{n}$ with $\cap_{i=1}^{m}C_{i}\neq\emptyset$. The CSVIP is
formulated as follows:
\begin{align}
\text{find a point }x^{\ast}  &  \in\cap_{i=1}^{m}C_{i}\text{ such that
}\left\langle f_{i}(x^{\ast}),x-x^{\ast}\right\rangle \geq0\text{ }\nonumber\\
\text{for all }x  &  \in C_{i}\text{, }i=1,2,\ldots,m.
\end{align}
This problem can be transformed into a CVIP in an appropriate product space
(different from the one in Section \ref{sec:SVIP}). Let $R^{mn}$ be the
product space and define $\boldsymbol{F}:R^{mn}\rightarrow R^{mn}$ by%

\begin{equation}
\boldsymbol{F}\left(  (x^{1},x^{2},\ldots,x^{m})\right)  =(f_{1}(x^{1}%
),\ldots,f_{m}(x^{m})),
\end{equation}
where $x^{i}\in R^{n}$ for all $i=1,2,\ldots,m.$ Let the diagonal set in
$R^{mn}$ be
\begin{equation}
\boldsymbol{\Delta}:=\{\boldsymbol{x}\in R^{mn}\mid\boldsymbol{x}%
\mathbf{=}(a,a,\ldots,a),\text{ }a\in R^{n}\}
\end{equation}
and define the product set%
\begin{equation}
\boldsymbol{C}:=\Pi_{i=1}^{m}C_{i}.
\end{equation}
The CSVIP in $R^{n}$ is equivalent to the following CVIP in $R^{mn}$:%
\begin{align}
\text{find a point }\boldsymbol{x}^{\ast}  &  \in\boldsymbol{C}\cap
\boldsymbol{\Delta}\text{ such that }\left\langle \boldsymbol{F}%
(\boldsymbol{x}^{\ast}),\boldsymbol{x-x}^{\ast}\right\rangle \geq0\text{
}\nonumber\\
\text{for all }\boldsymbol{x}  &  =(x^{1},x^{2},\ldots,x^{m})\in
\boldsymbol{C}.
\end{align}
So, this problem can be solved by using Algorithm \ref{alg:SubExt4SVIP} with
$\Omega=\boldsymbol{\Delta}.$ A new algorithm specifically designed for the
CSVIP appears in \cite{cgrs10}.

\subsection{The split minimization and the split zeros problems}

From optimality conditions for convex optimization (see, e.g., Bertsekas and
Tsitsiklis \cite[Proposition 3.1, page 210]{BT}) it is well-known that if
$F:R^{n}\rightarrow R^{n}$ is a continuously differentiable convex function on
a closed and convex subset $X\subseteq R^{n},$ then $x^{\ast}\in X$ minimizes
$F$ over $X$ if and only if%
\begin{equation}
\langle\nabla F(x^{\ast}),x-x^{\ast}\rangle\geq0\text{ for all }x\in X,
\label{eq:bert}%
\end{equation}
where $\nabla F$ is the gradient of $F$. Since (\ref{eq:bert}) is a VIP, we
make the following observation. If $F:R^{n}\rightarrow R^{n}$ and
$G:R^{m}\rightarrow R^{m}$ are continuously differentiable convex functions on
closed and convex subsets $C\subseteq R^{n}$ and $Q\subseteq R^{m},$
respectively, and if in the SVIP we take $f=\nabla F$ and $g=\nabla G,$ then
we obtain the following \textit{Split Minimization Problem }(SMP):%

\begin{gather}
\text{find a point }x^{\ast}\in C\text{ such that }x^{\ast}%
=\operatorname{argmin}\{f(x)\mid x\in C\}\\
\text{and such that}\nonumber\\
\text{the point }y^{\ast}=Ax^{\ast}\in Q\text{ and solves }y^{\ast
}=\operatorname{argmin}\{g(y)\mid y\in Q\}.
\end{gather}

The \textit{Split Zeros Problem} (SZP), newly introduced here, is defined as
follows. Let $H_{1}$ and $H_{2}$ be two Hilbert spaces. Given operators
$B_{1}:H_{1}\rightarrow H_{1}$ and $B_{2}:H_{2}\rightarrow H_{2},$ and a
bounded linear operator $A:H_{1}\rightarrow H_{2}$, the SZP is formulated as
follows:%
\begin{equation}
\text{find a point }x^{\ast}\in H_{1}\text{ such that }B_{1}(x^{\ast})=0\text{
and }B_{2}(Ax^{\ast})=0. \label{eq:SZP}%
\end{equation}
This problem is a special\ case of the SVIP if $A$ is a surjective operator.
To see this, take in (\ref{eq:vip})--(\ref{eq:svip}) $C=H_{1}$, $Q=H_{2},$
$f=B_{1}$ and $g=B_{2},$ and choose $x:=x^{\ast}-B_{1}(x^{\ast})\in H_{1}$ in
(\ref{eq:vip}) and $x\in H_{1}$ such that $Ax:=Ax^{\ast}-B_{2}(Ax^{\ast})\in
H_{2}$ in (\ref{eq:svip}).

The next lemma shows when the only solution of an SVIP is a solution of an
SZP. It extends a similar result concerning the relationship between the
(un-split) zero finding problem and the VIP.

\begin{lemma}
Let $H_{1}$ and $H_{2}$ be real Hilbert spaces, and $C\subseteq H_{1}$ and
$Q\subseteq H_{2}$ nonempty, closed and convex subsets. Let $B_{1}%
:H_{1}\rightarrow H_{1}$ and $B_{2}:H_{2}\rightarrow H_{2}$ be $\alpha$-ISM
operators and let $A:H_{1}\rightarrow H_{2}$ be a bounded linear operator.
Assume that $C\cap\{x\in H_{1}\mid B_{1}(x)=0\}\neq\emptyset$ and that
$Q\cap\{y\in H_{2}\mid B_{2}(y)=0\}\neq\emptyset$, and denote%
\begin{equation}
\Gamma:=\Gamma(C,Q,B_{1},B_{2},A):=\left\{  z\in SOL(C,B_{1})\mid Az\in
SOL(Q,B_{2})\right\}  .
\end{equation}
Then, for any $x^{\ast}\in C$ with $Ax^{\ast}\in Q,$ $x^{\ast}$ solves
(\ref{eq:SZP}) if and only if $x^{\ast}\in\Gamma$.
\end{lemma}

\begin{proof}
First assume that $x^{\ast}\in C$ with $Ax^{\ast}\in Q$ and that $x^{\ast}$
solves (\ref{eq:SZP}). Then it is clear that $x^{\ast}\in\Gamma.$ In the other
direction, assume that $x^{\ast}\in C$ with $Ax^{\ast}\in Q$ and that
$x^{\ast}\in\Gamma.$ Applying (\ref{eq:ProjP2}) with $C$ as $D$ there,
$(I-\lambda B_{1})\left(  x^{\ast}\right)  \in H_{1},$ for any $\lambda
\in(0,2\alpha]$, as $x$ there, and $q_{1}\in C\cap\operatorname*{Fix}%
(I-\lambda B_{1}),$ with the same $\lambda,$ as $y$ there, we get%
\begin{align}
&  \left\Vert q_{1}-P_{C}(I-\lambda B_{1})\left(  x^{\ast}\right)  \right\Vert
^{2}+\left\Vert (I-\lambda B_{1})\left(  x^{\ast}\right)  -P_{C}(I-\lambda
B_{1})\left(  x^{\ast}\right)  \right\Vert ^{2}\nonumber\\
&  \leq\left\Vert (I-\lambda B_{1})\left(  x^{\ast}\right)  -q_{1}\right\Vert
^{2},
\end{align}
and, similarly, applying (\ref{eq:ProjP2}) again, we obtain%
\begin{align}
&  \left\Vert q_{2}-P_{Q}(I-\lambda B_{2})\left(  Ax^{\ast}\right)
\right\Vert ^{2}+\left\Vert (I-\lambda B_{2})\left(  Ax^{\ast}\right)
-P_{Q}(I-\lambda B_{2})\left(  Ax^{\ast}\right)  \right\Vert ^{2}\nonumber\\
&  \leq\left\Vert (I-\lambda B_{2})\left(  Ax^{\ast}\right)  -q_{2}\right\Vert
^{2}.
\end{align}
Using the characterization of (\ref{eq:fix-vip}), we get%
\begin{equation}
\left\Vert q_{1}-x^{\ast}\right\Vert ^{2}+\left\Vert (I-\lambda B_{1})\left(
x^{\ast}\right)  -x^{\ast}\right\Vert ^{2}\leq\left\Vert (I-\lambda
B_{1})\left(  x^{\ast}\right)  -q_{1}\right\Vert ^{2}%
\end{equation}
and%
\begin{equation}
\left\Vert q_{2}-Ax^{\ast}\right\Vert ^{2}+\left\Vert (I-\lambda B_{2})\left(
Ax^{\ast}\right)  -x^{\ast}\right\Vert ^{2}\leq\left\Vert (I-\lambda
B_{2})\left(  Ax^{\ast}\right)  -q_{2}\right\Vert ^{2}.
\end{equation}

It can be seen from the proof of Lemma \ref{lemma:Mod-proj}(i) that the
operators $I-\lambda B_{1}$ and $I-\lambda B_{2}$ are nonexpansive for every
$\lambda\in\lbrack0,2\alpha]$, so with $q_{1}\in C\cap\operatorname*{Fix}%
(I-\lambda B_{1})$ and $q_{2}\in Q\cap\operatorname*{Fix}(I-\lambda B_{2}),$%
\begin{equation}
\left\Vert (I-\lambda B_{1})\left(  x^{\ast}\right)  -q_{1}\right\Vert
^{2}\leq\left\Vert x^{\ast}-q_{1}\right\Vert ^{2}%
\end{equation}
and%
\begin{equation}
\left\Vert (I-\lambda B_{2})\left(  Ax^{\ast}\right)  -q_{2}\right\Vert
^{2}\leq\left\Vert Ax^{\ast}-q_{2}\right\Vert ^{2}.
\end{equation}
Combining the above inequalities, we obtain%
\begin{equation}
\left\Vert q_{1}-x^{\ast}\right\Vert ^{2}+\left\Vert (I-\lambda B_{1})\left(
x^{\ast}\right)  -x^{\ast}\right\Vert ^{2}\leq\left\Vert x^{\ast}%
-q_{1}\right\Vert ^{2}%
\end{equation}
and%
\begin{equation}
\left\Vert q_{2}-Ax^{\ast}\right\Vert ^{2}+\left\Vert (I-\lambda B_{2})\left(
Ax^{\ast}\right)  -x^{\ast}\right\Vert ^{2}\leq\left\Vert Ax^{\ast}%
-q_{2}\right\Vert ^{2}.
\end{equation}
Hence, $\left\Vert (I-\lambda B_{1})\left(  x^{\ast}\right)  -x^{\ast
}\right\Vert ^{2}=0$ and $\left\Vert (I-\lambda B_{2})\left(  Ax^{\ast
}\right)  -Ax^{\ast}\right\Vert ^{2}=0.$ Since $\lambda>0,$ we get that
$B_{1}(x^{\ast})=0$ and $B_{2}(Ax^{\ast})=0$, as claimed$\medskip$
\end{proof}

\textbf{Acknowledgments}. This work was partially supported by a United
States-Israel Binational Science Foundation (BSF) Grant number 200912, by US
Department of Army award number W81XWH-10-1-0170, by Israel Science Foundation
(ISF) Grant number 647/07, by the Fund for the Promotion of Research at the
Technion and by the Technion President's Research Fund.

\end{document}